\documentclass[11pt,reqno]{amsart}
\usepackage{graphicx}
\usepackage{amsmath,amssymb,mathrsfs}
\usepackage[usenames]{color}

\newtheorem{thm}{Theorem}[section]

\newtheorem{lem}[thm]{Lemma}

\theoremstyle{definition}

\theoremstyle{remark}
\newtheorem{rem}[thm]{Remark}
\numberwithin{equation}{section}

\begin{document}

\title[Approximate Acoustic Cloaking]{On near-cloak in acoustic scattering}

\author{Hongyu Liu}
\address{Department of
Mathematics and Statistics, University of North Carolina, Charlotte,
NC 28223, USA.}
\email{hongyu.liuip@gmail.com}

\thanks{\emph{2010 Mathematics Subject Classification.} 35Q60, 35J05, 31B10, 35R30, 78A40}

\keywords{Inverse acoustic scattering, invisibility cloaking, transformation optics, lossy layer, regularization, asymptotic estimates}

\date{}

\maketitle

\begin{abstract}
Invisibility cloaking in acoustic scattering via the approach of transformation optics is considered. The near-cloaks of both passive medium and active/radiating object are investigated. From a practical viewpoint, we are especially interested in the cloaking of an arbitrary (but regular) content. It is shown that one cannot achieve the near-cloak unless some special mechanism is introduced into the construction. A general lossy layer is incorporated into the construction of our near-cloaking devices. We derive very accurate estimates of the scattering amplitude in terms of the regularization parameter and the material parameters of the lossy layer in different settings.  
\end{abstract}

\section{Introduction}

We shall be concerned with the invisibility cloaking in acoustic scattering. Let $D$ be a bounded convex $C^2$ domain in $\mathbb{R}^N$, $N\geq 2$. We assume that $D$ contains the origin and let
\[
D_\rho=\{\rho x; x\in D\},\quad \rho\in\mathbb{R}_+.
\]
Let $\Omega$ be a bounded domain in $\mathbb{R}^N$ such that $\Omega^c:=\mathbb{R}^N\backslash\overline{\Omega}$ is connected and $D\subset\hspace*{-0.3mm}\subset\Omega$. Consider the following scattering problem due to an inhomogeneity supported in $\Omega$ and a radiating source $f$,
\begin{equation}\label{eq:Helm1}
\begin{cases}
& \displaystyle{\sum_{i,j=1}^N \frac{\partial}{\partial x_i}\left(\sigma^{ij}(x)\frac{\partial u(x)}{\partial x_j} \right)+\omega^2 q(x) u(x)=f(x)},\quad x\in \mathbb{R}^N,\\
& \displaystyle{u(x)=u^s(x)+u^{i}(x;\omega),\hspace*{1cm} x\in \Omega^c,}\\
& \displaystyle{\lim_{|x|\rightarrow\infty}|x|^{(N-1)/2}\left\{\frac{\partial u^s(x)}{\partial |x|}-i\omega u^s(x)\right\}=0},
\end{cases}
\end{equation}
where
%
\begin{equation}\label{eq:medium}
\sigma(x), q(x)=\begin{cases}
I, 1\quad & \mbox{in\ \ $\Omega^c$},\\
\sigma_c, q_c\quad & \mbox{in\ \ $\Omega\backslash D$},\\
\sigma_l, q_l\quad & \mbox{in $D\backslash D_{1/2}$},\\
\sigma_a, q_a\quad & \mbox{in\ \ $D_{1/2}$},
\end{cases}
\end{equation}
and $supp(f)\subset D_{1/2}$ with $f\in L^2(D_{1/2})$. In \eqref{eq:Helm1}, $u^i(x;\omega)$ is an entire solution to $\Delta v+\omega^2 v=0$ in the whole space. In the physical situation, (\ref{eq:Helm1}) describes the scattering of an inhomogeneous medium supported in $\Omega$ and a radiating source $f$ supported in $D_{1/2}$ due to a time-harmonic wave $u^{i}(x;\omega)$ oscillating with frequency $\omega\in\mathbb{R}_+$. In the classical scattering theory, we note that $u^i$ is usually taken to be the plane wave $e^{i\omega x\cdot d}$ with $d\in\mathbb{S}^{N-1}$. $\sigma$ and $q$ are the acoustical material parameters with $\sigma^{-1}$ denoting the density tensor and $q$ the modulus. We assume that $q\in L^\infty(\mathbb{R}^N)$ with $\Im q\geq 0$ and, $\sigma=(\sigma^{ij})_{i,j=1}^N$ is a symmetric matrix and uniformly elliptic in the sense that
\[
c|\xi|^2\leq\sum_{i,j=1}^N \sigma(x)^{ij}\xi_i\xi_j\leq C|\xi|^2, \quad \forall \xi\in\mathbb{R}^N,\ \ \forall x\in\mathbb{R}^N,
\]
where $c$ and $C$ are positive constants. In our subsequent study, an acoustic medium is referred to as {\it regular} if its modulus parameter $q$ is essentially bounded with $\Im q\geq 0$ and its density tensor is symmetric and uniformly elliptic. It is known that (\ref{eq:Helm1}) has a unique solution $u\in H^1_{loc}(\mathbb{R}^N)$ admitting the following asymptotic development as $|x|\rightarrow +\infty$(cf. \cite{ColKre,Isa,Mcl,Ned} )
\begin{equation}\label{eq:farfield}
u(x)=u^i+\frac{e^{i\omega |x|}}{|x|^{(N-1)/2}}A(\hat{x})+\mathcal{O}\left(\frac{1}{|x|^{(N+1)/2}}\right),\quad |x|\rightarrow \infty,
\end{equation}
where $\hat{x}=x/|x|\in\mathbb{S}^{N-1}$ for $x\in\mathbb{R}^N$. $A(\hat{x})$ is known as the {\it scattering amplitude} or the {\it far-field pattern}, which encodes the exterior wave patterns produced by the underlying scattering object. The classical inverse scattering problem of significant practical importance is to recover the inhomogeneity $\sigma, q$ and/or the radiating source $f$ from the measurement of the corresponding scattering amplitude $A(\hat{x})$. In this work, we are mainly interested in the setting that $\sigma_a, q_a$ and $f$ supported in $D_{1/2}$ are the target objects and, $(\sigma_l, q_l)$ in $D\backslash D_{1/2}$ and $(\sigma_c, q_c)$ in $\Omega\backslash D$ are some designed cloaking medium which could hide the target object from exterior wave detections. Our construction is based on the so-called {\it transformation optics} (cf. \cite{GLU},\cite{GLU2},\cite{Leo},\cite{PenSchSmi}). Throughout we shall suppose that there exists a (uniformly) bi-Lipschitz and orientation-preserving map,
\begin{equation}\label{eq:transformation1}
F_\varepsilon^{(1)}:\ \ \overline{\Omega}\backslash D_\varepsilon\rightarrow \overline{\Omega}\backslash D,\quad F_\varepsilon^{(1)}|_{\partial\Omega}=\mbox{Identity},
\end{equation}
where $\varepsilon\in\mathbb{R}_+$ and $0<\varepsilon<1$. Let
\begin{equation}\label{eq:transformation2}
F_\varepsilon^{(2)}(x)=\frac{x}{\varepsilon},\quad x\in D_\varepsilon,
\end{equation}
and
\begin{equation}\label{eq:transformation}
F_\varepsilon=\begin{cases}
\mbox{Identity}\quad & \mbox{on\ \ $\Omega^c$},\\
F_\varepsilon^{(1)}\quad & \mbox{on\ \ $\overline{\Omega}\backslash D_\varepsilon$},\\
F_\varepsilon^{(2)}\quad & \mbox{on\ \ $D_\varepsilon$}. 
\end{cases}
\end{equation}
Set
\begin{equation}\label{eq:pushc}
\sigma_c=(F_\varepsilon^{(1)})_*\sigma_0,\quad q_c=(F_\varepsilon^{(1)})_*q_0,
\end{equation}
where $(\sigma_0, q_0)=(I, 1)$ and the push-forwards are defined by
\begin{equation}\label{eq:pushforward}
\begin{split}
&(F_\varepsilon^{(1)})_*\sigma_0(y)=\frac{DF_\varepsilon^{(1)}(x)\cdot \sigma_0(x)\cdot (DF_\varepsilon^{(1)}(x))^T}{\left|\det(DF_\varepsilon^{(1)}(x))\right|}\bigg|_{x=(F_\varepsilon^{(1)})^{-1}(y)},\\
& (F_\varepsilon^{(1)})_*q_0(y)=\frac{q_0(x)}{\left|\det(DF_\varepsilon^{(1)}(x))\right|}\bigg|_{x=(F_\varepsilon^{(1)})^{-1}(y)}, \quad x\in\Omega\backslash D_\varepsilon,\ \  y\in\Omega\backslash D,
\end{split}
\end{equation}
where $DF_{\varepsilon}^{(1)}$ denotes the Jacobian matrix of $F_\varepsilon^{(1)}$.
In a similar manner, we set
\begin{equation}\label{eq:lossy push}
\sigma_l=(F_\varepsilon^{(2)})_*\widetilde{\sigma}_l
\end{equation}
with
\begin{equation}\label{eq:lossy virtual}
\widetilde{\sigma}_l(x)=\varepsilon^r(\gamma(x/\varepsilon)Pr(n(x'))+g(x/\varepsilon)(I-Pr(n(x')))),\quad x\in D_\varepsilon\backslash D_{\varepsilon/2}
\end{equation}
where $r>2-N/2$ and $\gamma(x)\in C^2(\overline{D}\backslash D_{1/2})$ is a positive function that is bounded below, $g(x)=(g^{ij}(x))_{i,j=1}^N\in C^2(\overline{D}\backslash D_{1/2})^{N\times N}$ is symmetric and uniformly elliptic, and $x'\in\partial D$ lying on the line passing through the origin and $x\in D_\varepsilon\backslash D_{\varepsilon/2}$, and $n(x')\in\mathbb{S}^{N-1}$ is the exterior unit normal vector to $\partial D$ at $x'$ and $Pr(n(x'))$ denotes the projection along the $n(x')$-direction. Moreover, we let
\begin{equation}\label{eq:lossy modulus}
q_l=(F_\varepsilon^{(2)})_*\widetilde{q}_l\quad \mbox{with}\quad \widetilde{q}_l(x)=\alpha(x/\varepsilon)+i\beta(x/\varepsilon),\quad x\in D_\varepsilon\backslash D_{\varepsilon/2}
\end{equation}
where $\alpha(x)$ and $\beta(x)$ for $x\in D\backslash D_{1/2}$ are positive functions that are bounded below and above. We shall show that the above construction will produce a practical near-cloaking device; that is, the corresponding scattering amplitude will be asymptotically small in terms of the accurate quantitative estimates derived in the subsequent section.

In recent years, the transformation optics and invisibility cloaking have received significant attentions;  see, e.g., \cite{CC,GKLU4,GKLU5,Nor} and references therein. The ideal cloaking requires singular cloaking material and in order to avoid the singular structure, various near-cloaking schemes have been developed in literature. Some earlier results on approximate cloaking were presented in \cite{GKLUoe,GKLU_2,RYNQ}, where truncation of singularities are considered for spherical cloaking devices with uniform cloaked contents. In \cite{KOVW,KSVW}, the `blow-up-a-small-region' construction were proposed, and the authors show that in order to cloak an arbitrary target medium, one has to employ a lossy layer right between the cloaking region and the cloaked region, otherwise there is cloak-busting inclusion which defies the attempt to cloak. In \cite{KSVW}, the construction is similar to the one in (\ref{eq:medium}), but that the material parameters of the lossy layer are given by
\begin{equation}\label{eq:lossy1}
\sigma_l=\varepsilon^{N-2}I,\quad q_l=\varepsilon^N(1+ic_0\varepsilon^{-2})\quad \mbox{in\ \ $D\backslash D_{1/2}$},
\end{equation}
where $c_0$ is a positive constant. In \cite{LiLiuSun,LiuSun} a different lossy layer is proposed by employing a high density medium instead of high loss, 
\begin{equation}\label{eq:lossy2}
\sigma_l=\varepsilon^{N+\delta} I,\quad q_l=\varepsilon^N(a_0+i b_0)\quad \mbox{in\ \ $D\backslash D_{1/2}$},
\end{equation}
where $a_0$, $b_0$ and $\delta$ are positive constants. Both the cloaking schemes in \cite{KSVW} and \cite{LiuSun} are assessed in terms of exterior boundary measurements encoded into the boundary Neumann-to-Dirichlet map, and the one in \cite{LiuSun} is shown to produce enhanced performance than existing ones. Moreover, both the schemes in \cite{KSVW} and \cite{LiuSun} are assessed for the cloaking of passive media only.  More approximate cloaking schemes of different nature could be found in 
\cite{Ammari1,Ammari2,Ammari3,Liu}. By straightforward calculations, one can show that material parameters in (\ref{eq:lossy push}) and (\ref{eq:lossy modulus}) for our present study are given by
\begin{equation}\label{eq:lossy effective 1}
{\sigma}_l(x)=\varepsilon^{N+r-2}(\gamma(x)Pr(n(x'))+g(x)(I-Pr(n(x')))),\quad x\in D\backslash D_{1/2}
\end{equation}
and
\begin{equation}\label{eq:lossy effective 2}
q_l(x)=\varepsilon^N(\alpha(x)+i\beta(x)),\quad x\in D\backslash D_{1/2}.
\end{equation}
Compared to the lossy layer \eqref{eq:lossy1} and \eqref{eq:lossy2}, our present construction \eqref{eq:lossy effective 1}--\eqref{eq:lossy effective 2} clearly has more flexibility. It allows the lossy layer to be variable or even anisotropic, and this would be of practical importance when fabrication fluctuation happens. Moreover, comparing \eqref{eq:lossy2} and \eqref{eq:lossy effective 2}, we see that the density in \eqref{eq:lossy effective 1} could be less high than that of \eqref{eq:lossy2}, but we will show that one could still achieve favorable near-cloaks. By comparing \eqref{eq:lossy1} and \eqref{eq:lossy effective 2}, especially by taking $N=2$, we see that in \eqref{eq:lossy1} one has to implement a high lossy parameter, whereas in \eqref{eq:lossy effective 1}--\eqref{eq:lossy effective 2} one could make use of a layer with low loss but with a reasonably high density. For the construction \eqref{eq:medium} with \eqref{eq:pushc}, \eqref{eq:lossy effective 1} and \eqref{eq:lossy effective 2}, we show that one could achieve the near-cloak disregarding the passive content $(\sigma_a, q_a)$, i.e. it could be arbitrary. This is of critical importance from a practical viewpoint. In addition to the cloaking of a passive medium, we also consider the cloaking of an active/radiating source for our construction. By the example given in the next section, one can see that if the passive content is allowed to be arbitrary (but regular), one cannot cloak a generic source without a lossy layer. However,  if the cloaked region is maintained to be absorbing, we show that one can achieve a much practical and favorable near-cloak of a source term by our scheme.

The rest of the paper is organized as follows. In Section 2, we present the main theorems in assessing the near-cloaking performances of our construction in different settings. Section 3 is devoted to the conclusion and discussion.
\section{Main results}

We first show that in order to assess the near-cloaking performance of our construction, the study could be reduced to estimating the scattering amplitude due to a small inclusion with arbitrary contents enclosed by a thin lossy layer. We shall make use of the following theorem collecting the key ingredient of transformation optics.

\begin{thm}\label{thm:transoptics}
Let $\Omega$ and $\widetilde{\Omega}$ be two bounded Lipschitz domains in $\mathbb{R}^N$ and suppose that there exists a bi-Lipschitz and orientation-preserving mapping $F:\Omega\rightarrow \widetilde{\Omega}$. Let $u\in H^1(\Omega)$ satisfy 
\[
\nabla\cdot(\sigma(x)\nabla u(x))+\omega^2 q(x) u(x)=f(x),\quad x\in \Omega,
\]
where $\sigma(x), q(x)$, $x\in\Omega$ are uniformly elliptic and $f\in L^2(\Omega)$. Then one has that $\widetilde{u}=(F^{-1})^* u:=u\circ F^{-1}\in H^1(\widetilde{\Omega})$ satisfies
\[
\nabla\cdot(\widetilde{\sigma}(y)\nabla \widetilde{u}(y))+\omega^2 \widetilde{q}(y) \widetilde{u}(y)=\widetilde{f}(y), \ y\in \widetilde{\Omega},
\]
where
\[
\widetilde{\sigma}=F_*\sigma, \ \ \widetilde{q}=F_*q,\ \ \widetilde{f}=\left(\frac{f}{\left|\det(DF)\right|}\right)\circ F^{-1}.
\]
\end{thm}
The proof of Theorem~\ref{thm:transoptics} could be found in \cite{GKLU5,KOVW,Liu}. By using Theorem~\ref{thm:transoptics}, we have by direct verification that
\begin{lem}\label{lem:reduce}
Let $u\in H_{loc}^1(\mathbb{R}^N)$ be the solution to (\ref{eq:Helm1}) with $(\sigma_c, q_c)$ given in (\ref{eq:pushc}) and $(\sigma_l, q_l)$ given in (\ref{eq:lossy push})--(\ref{eq:lossy modulus}), then $u_\varepsilon=F_\varepsilon^*u\in H_{loc}^1(\mathbb{R}^N)$ is the solution to
\begin{equation}\label{eq:virtual Helm1}
\begin{cases}
& \nabla\cdot(\sigma_\varepsilon(x)\nabla u_\varepsilon(x))+\omega^2 q_\varepsilon(x) u_\varepsilon(x)=f_\varepsilon(x),\quad x\in\mathbb{R}^N,\\
& u_\varepsilon(x)=u_\varepsilon^s(x)+u^i(x;\omega),\qquad x\in \mathbb{R}^N\backslash \overline{D_\varepsilon},\\
& \displaystyle{\lim_{|x|\rightarrow\infty}|x|^{(N-1)/2}\left\{\frac{\partial u_\varepsilon^s(x)}{\partial |x|}-i\omega u_\varepsilon^s(x)\right\}=0},
\end{cases}
\end{equation}
where
\begin{equation}\label{eq:virtual medium}
\sigma_\varepsilon(x), q_\varepsilon(x)=\begin{cases}
I, 1\quad & \mbox{in\ \ $(D_\varepsilon)^c$},\\
\widetilde{\sigma}_l, \widetilde{q}_l\quad & \mbox{in $D_\varepsilon\backslash D_{\varepsilon/2}$},\\
\widetilde{\sigma}_a, \widetilde{q}_a\quad & \mbox{in\ \ $D_{\varepsilon/2}$},
\end{cases}
\end{equation}
with $\widetilde{\sigma}$ and $\widetilde{q}_l$ given, respectively, in (\ref{eq:lossy virtual}) and (\ref{eq:lossy modulus}), and $f_\varepsilon=\varepsilon^{-N}F_\varepsilon^* f$, and
\begin{equation}\label{eq:virtual target}
\widetilde{\sigma}_a(x)=\varepsilon^{2-N}\sigma_a(x/\varepsilon),\quad \widetilde{q}_a(x)=\varepsilon^{-N}q_a(x/\varepsilon),\quad x\in D_{\varepsilon/2}.
\end{equation}
\end{lem}

In the sequel, we let $A_\varepsilon(\hat{x})$ denote the scattering amplitude corresponding to $u_\varepsilon$.
Since $u=u_\varepsilon$ in $\Omega^c$, one sees that
\begin{equation}\label{eq:eqivalence}
A(\hat{x})=A_\varepsilon(\hat{x}),\quad \hat{x}\in\mathbb{S}^{N-1}.
\end{equation}
Hence, in order assess the near-cloaking performance, i.e., in order to evaluate the scattering amplitude $A(\hat{x})$ to the physical problem (\ref{eq:Helm1}), it suffices for us to evaluate the scattering amplitude $A_\varepsilon(\hat{x})$ to the virtual problem (\ref{eq:virtual Helm1}).
Throughout the rest of the paper, we let
\begin{equation}\label{eq:scattered}
u_\varepsilon^-=u_\varepsilon|_{D_\varepsilon}\in H^1(D_\varepsilon),\quad u_\varepsilon^+=u_\varepsilon|_{\mathbb{R}^N\backslash\overline{D}_\varepsilon}\in H_{loc}^1(\mathbb{R}^N\backslash\overline{D}_\varepsilon)
\end{equation}
and 
\begin{equation}\label{eq:s2}
u_\varepsilon^s=u_\varepsilon^+-u^i\in H_{loc}^1(\mathbb{R}^N\backslash\overline{D}_\varepsilon).
\end{equation}
By straightforward calculations, it can be shown from (\ref{eq:virtual Helm1}) that
\begin{equation}\label{eq:Helm2}
\begin{cases}
& \Delta u_\varepsilon^s+\omega^2 u_\varepsilon^s=0\hspace*{4.05cm} \mbox{in\ \ \ $\mathbb{R}^N\backslash \overline{D}_\varepsilon$},\\
& \nabla\cdot(\widetilde{\sigma}_l\nabla u_\varepsilon^-)+\omega^2 \widetilde{q}_l u_\varepsilon^-=0\hspace*{2.38cm}\mbox{in\ \ \ $D_\varepsilon\backslash\overline{D}_{\varepsilon/2}$},\\
& \nabla\cdot(\widetilde{\sigma}_a\nabla u_\varepsilon^-)+\omega^2 \widetilde{q}_a u_\varepsilon^-=f_\varepsilon\hspace*{2.12cm}\mbox{in\ \ \ $D_{\varepsilon/2}$},\\
& u_\varepsilon^s=u_\varepsilon^--u^i \hspace*{4.55cm} \mbox{on\ \ \ $\partial D_\varepsilon$},\\
& \displaystyle{ \frac{\partial u_\varepsilon^s}{\partial n}=\sum_{i,j=1}^N n_i\widetilde{\sigma}_l^{ij}\partial_j u_\varepsilon^-}-\frac{\partial u^i}{\partial n}\hspace*{2.0cm} \mbox{on\ \ \ $\partial D_\varepsilon$},\\
 &\displaystyle{\lim_{|x|\rightarrow\infty}|x|^{(N-1)/2}\left\{\frac{\partial u_\varepsilon^s}{\partial |x|}-i\omega u_\varepsilon^s\right\}=0},
\end{cases}
\end{equation}
where $n(x)=(n_i(x))_{i=1}^N$ is the exterior unit normal vector to $\partial D_\varepsilon$. Obviously, $A_\varepsilon(\hat{x})$ could be read off from the large $|x|$ asymptotics of $u_\varepsilon^s(x)$.

We shall first consider the cloaking of passive medium, i.e., $f=0$ in (\ref{eq:Helm1}) (correspondingly, $f_\varepsilon=0$ in (\ref{eq:virtual Helm1}) and (\ref{eq:Helm2})).  

\begin{thm}\label{thm:main1}
Suppose $f=0$ and $\sigma_a, q_a$ are arbitrary but regular. Let $A_\varepsilon(\hat{x})$ be the scattering amplitude to (\ref{eq:Helm2}). Let $B_R$, $R\in\mathbb{R}_+$, be a central ball of radius $R$ such that $\Omega\subset B_R$. Then there exists $\varepsilon_0\in\mathbb{R}_+$ such that when $\varepsilon<\varepsilon_0$
\begin{equation}\label{eq:estimate main 1}
|A_\varepsilon(\hat{x})|\leq C \varepsilon^{\min\{N+2r-4, N\}}\|u^i\|_{H^1(B_R)},\quad \forall \hat{x}\in\mathbb{S}^{N-1},
\end{equation}
where $C$ is positive constant independent of $\varepsilon, r$, $\sigma_a$ and $q_a$.
\end{thm}

\begin{rem}\label{rem:passive}
By Theorem~\ref{thm:main1} and (\ref{eq:eqivalence}), we see that our construction underlying (\ref{eq:Helm1}) produces an approximate cloaking device which is within $\varepsilon^{\min\{N+2r-4, N\}}$ of the perfect cloaking. Moreover, the estimate in (\ref{eq:estimate main 1}) indicates that one can cloak an arbitrary medium, which is of critical importance from a practical viewpoint. It is mentioned in Introduction that in \cite{KOVW}, the authors showed that there are cloak-busting inclusions which defy the attempt to achieve the near-cloak. Specifically, it is shown that no matter how small a region is, there exists certain inhomogeneity supported in that region such that it can produce significant wave scattering. For the boundary value problem considered in \cite{KOVW}, this is shown to be caused by resonances. Hence, a damping mechanism should be incorporated in order to defeat such `resonant' inclusions. In Remark~\ref{rem:active} in the following, one will see that in addition to the `resonant' inclusions, there are more `cloak-busting' inclusions if one intends to cloak an active content.   
\end{rem}

In order to prove Theorem~\ref{thm:main1}, we first derive the following lemma. Since the following lemma will also be needed in our subsequent study on cloaking of radiating/active objects, we would like to emphasize that it holds for the general case with $f$ not necessarily being zero.

\begin{lem}\label{lem:main11}
Suppose that 
\[
\alpha(x)\leq \overline{\alpha}_0\quad \mbox{and}\quad \underline{\beta}_0\leq \beta(x)\leq \overline{\beta}_0,\quad x\in D\backslash D_{1/2},
\]
where $\overline{\alpha}_0$, $\underline{\beta}_0$ and $\overline{\beta}_0$ are positive constants.
Let $u_\varepsilon\in H_{loc}^1(\mathbb{R}^N)$ be the solution to (\ref{eq:Helm2}). Then we have
\begin{align}
&\left\|\left(\sum_{i,j=1}^Nn_i\widetilde{\sigma}_l^{ij}\partial_j u_\varepsilon^-\right)(\varepsilon\ \cdot)\right\|_{H^{-3/2}(\partial (D\backslash \overline{D}_{1/2}))} \nonumber\\
\leq &\ \varepsilon^{r-1-N/2}\left(C+\sqrt{\overline{\alpha}_0^2+\overline{\beta}_0^2}\omega^2\varepsilon^{2-r}\right)\|u_\varepsilon^-\|_{L^2(D_\varepsilon\backslash D_{\varepsilon/2})},\label{eq:controlN}
\end{align}
where $C$ is a positive constant dependent only on $\gamma$, $g$ and $D$, but independent of $\varepsilon$, $r$, $\alpha$ and $\beta$. 
\end{lem}

\begin{proof}
We shall make use of the duality argument by noting that
\begin{equation}\label{eq:dual}
\begin{split}
&\left\|\left(\sum_{i,j=1}^Nn_i\widetilde{\sigma}_l^{ij}\partial_j u_\varepsilon^-\right)(\varepsilon\ \cdot)\right\|_{H^{-3/2}(\partial (D\backslash \overline{D}_{1/2}))} \\
=&\sup_{\|\varphi\|_{H^{3/2}(\partial (D\backslash \overline{D}_{1/2}) )}\leq 1}\left|\int_{\partial(D\backslash \overline{D}_{1/2})} \left(\sum_{i,j=1}^Nn_i\widetilde{\sigma}_l^{ij}\partial_j u_\varepsilon^-\right)(\varepsilon x) \varphi(x)\ ds_x \right|.
\end{split}
\end{equation}
For any $\varphi\in H^{3/2}(\partial(D\backslash D_{1/2}))$, there exists $w\in H^2(D\backslash\overline{D}_{1/2})$ such that
\begin{align*}
&(i)~~w=\varphi\ \ \mbox{on}\ \ \partial D\cup \partial D_{1/2};\\
&(ii)~~\frac{\partial w}{\partial n}=0\ \ \mbox{on}\ \ \partial D\cup\partial D_{1/2};\\
&(iii)~~\|w\|_{H^2(D\backslash\overline{D}_{1/2})}\leq C\|\varphi\|_{H^{3/2}(\partial (D\backslash \overline{D}_{1/2}))}.
\end{align*}
Let
\[
\widehat{\sigma}_l(x)=\varepsilon^{2-r}\sigma_l(x)=(\gamma(x)Pr(n(x'))+g(x)(I-Pr(n(x')))),\quad x\in D\backslash D_{1/2},
\]
where and in the following $\sigma_l, q_l$ are given in (\ref{eq:lossy effective 1})--(\ref{eq:lossy effective 2}). 
We first note that
\[
\sum_{i,j=1}^N n_i\widehat{\sigma}_l^{ij}\partial_j w=\gamma (n\cdot\nabla w)=0\quad\mbox{on\ \ $\partial D\cup\partial D_{1/2}$}.
\]
Set $v(x)=u_\varepsilon^{-}(\varepsilon x)$ for $x\in D_{\varepsilon}\backslash\overline{D}_{\varepsilon/2}$, and it is directly verfied that 
\begin{equation}\label{eq:t1}
\nabla\cdot(\widehat{\sigma}_l \nabla v)=-\omega^2\varepsilon^{2-r} q_l v\quad \mbox{in\ \ $D\backslash\overline{D}_{1/2}$}.
\end{equation}
Now, we have
\begin{align*}
& \left|\int_{\partial(D\backslash \overline{D}_{1/2})} \left(\sum_{i,j=1}^Nn_i\widetilde{\sigma}_l^{ij}\partial_j u_\varepsilon^-\right)(\varepsilon x) \varphi(x)\ ds_x\right| \\
=& \varepsilon^{r-1}\left|\int_{\partial (D\backslash\overline{D}_{1/2})} \left(\sum_{i,j=1}^N n_i\widehat{\sigma}_l^{ij}\partial_j v\right) w\  ds_x-\int_{\partial (D\backslash \overline{D}_{1/2})} \left(\sum_{i,j=1}^N n_i\widehat{\sigma}_l^{ij}\partial_j w\right) v\ ds_x \right|\\
=& \varepsilon^{r-1}\left|\int_{D\backslash\overline{D}_{1/2}} \nabla\cdot(\widehat{\sigma}_l\nabla v) w\ dx-\int_{D\backslash\overline{D}_{1/2}} \nabla\cdot(\widehat{\sigma}_l \nabla w) v\ dx\right|\\
=&\varepsilon^{r-1}\left|    \omega^2\varepsilon^{2-r}\int_{D\backslash\overline{D}_{1/2}} q_l v w\ dx+\int_{D\backslash\overline{D}_{1/2}} \nabla\cdot(\widehat{\sigma}_l\nabla w) v\ dx   \right|\\
\leq & \varepsilon^{r-1}\bigg[ \omega^2\varepsilon^{2-r}\sqrt{\overline{\alpha}_0^2+\overline{\beta}_0^2} \|w\|_{L^2(D\backslash D_{1/2})}\|v\|_{L^2(D\backslash D_{1/2})}\\
 &\ \ \ +\|\nabla\cdot(\widehat{\sigma}_l\nabla w)\|_{L^2(D\backslash\overline{D}_{1/2})}\|v\|_{L^2(D\backslash D_{1/2})}\bigg]\\
 \leq &\ \varepsilon^{r-1}  \left(C+\sqrt{\overline{\alpha}_0^2+\overline{\beta}_0^2}\omega^2\varepsilon^{2-r}\right)\|w\|_{H^2(D\backslash\overline{D}_{1/2})}\|v\|_{L^2(D\backslash D_{1/2})}\\
 \leq &\ \varepsilon^{r-1-N/2}\left(C+\sqrt{\overline{\alpha}_0^2+\overline{\beta}_0^2}\omega^2\varepsilon^{2-r}\right)\|u_\varepsilon^-\|_{L^2(D_\varepsilon\backslash D_{\varepsilon/2})} \|\varphi\|_{H^{3/2}(\partial D)},
\end{align*}
where in the last inequality we have made use of the fact that
\[
\|v\|_{L^2(D\backslash D_{1/2})}=\|u_\varepsilon^-(\varepsilon\cdot)\|_{L^2(D\backslash D_{1/2})}=\varepsilon^{-N/2}\|u_\varepsilon^-\|_{L^2(D_\varepsilon\backslash D_{\varepsilon/2})}.
\]
Hence, we have verified (\ref{eq:controlN}). 

The proof is completed.
\end{proof}

Next, let $B_R$, $R\in\mathbb{R}_+$, be the central ball in Theorem~\ref{thm:main1}. Without loss of generality, we assume that $\omega^2$ is not a Dirichlet eigenvalue for the negative Laplacian in $B_R$. Consider the Helmholtz equation,
\begin{equation}\label{eq:auxiliary}
\begin{cases}
& \Delta v+\omega^2 v=0\hspace*{2.35cm}\mbox{in\ \ $\mathbb{R}^N\backslash\overline{B}_R$},\\
& \displaystyle{\frac{\partial v}{\partial n}=\psi\in H^{-1/2}(\partial B_R)\qquad \mbox{on\ \ $\partial B_R$},}\\
& \displaystyle{\lim_{|x|\rightarrow\infty}|x|^{(N-1)/2}\left\{\frac{\partial v}{\partial |x|}-i\omega v\right\}=0}.
\end{cases}
\end{equation}
Define the Neumann-to-Dirichlet map by,
\[
\Lambda(\psi)=v|_{\partial B_R}\in H^{1/2}(\partial B_R),
\]
where $v\in H^1_{loc}(\mathbb{R}^N\backslash\overline{B}_R)$ is the unique solution to (\ref{eq:auxiliary}). It is known that $\Lambda$ is bounded and invertible from $H^{-1/2}(\partial B_R)$ to $H^{1/2}(\partial B_R)$ (see \cite{Mcl} and \cite{Ned}).

\begin{lem}\label{lem:main12}
Suppose that 
\[
\alpha(x)\leq \overline{\alpha}_0\quad \mbox{and}\quad \underline{\beta}_0\leq \beta(x)\leq \overline{\beta}_0,\quad x\in D\backslash D_{1/2},
\]
where $\overline{\alpha}_0$, $\underline{\beta}_0$ and $\overline{\beta}_0$ are positive constants.
Let $u_\varepsilon\in H_{loc}^1(\mathbb{R}^N)$ be the solution to (\ref{eq:Helm2}) with $f_\varepsilon=0$. Then we have
\begin{equation}\label{eq:l2 estimate}
\begin{split}
&\omega^2\underline{\beta}_0\int_{D_\varepsilon\backslash D_{\varepsilon/2}} |u_\varepsilon^-|^2\ d x \leq\ C\bigg( \left\|\Lambda\left(\frac{\partial u_\varepsilon^s}{\partial n}\right)\right\|_{H^{1/2}(\partial B_R)} \left\| \frac{\partial u^i}{\partial n} \right\|_{H^{-1/2}(\partial B_R)} \\
&+ \left\|\frac{\partial u_\varepsilon^s}{\partial n}\right\|_{H^{-1/2}(\partial B_R)}\left\| u^i \right\|_{H^{1/2}(\partial B_R)}+\varepsilon^N\|u^i\|_{H^1(B_R)}^2\\
& +\left\|\frac{\partial u_\varepsilon^s}{\partial n}\right\|_{H^{-1/2}(\partial B_R)} \left\|\Lambda\left(\frac{\partial u_\varepsilon^s}{\partial n}\right)\right\|_{H^{1/2}(\partial B_R)}\bigg),
\end{split}
\end{equation}
where $C$ is a positive constant depending only on $R$ and $\omega$. 
\end{lem}

\begin{proof}
By multiplying both sides of the first equation in (\ref{eq:s2}) by $\overline{u_\varepsilon^s}$ and by using integration by parts over $B_R\backslash\overline{D}_\varepsilon$, we have
\begin{equation}\label{eq:d1}
\begin{split}
&-\int_{B_R\backslash\overline{D}_\varepsilon} |\nabla u_\varepsilon^s|^2\ dx+\omega^2\int_{B_R\backslash\overline{D_\varepsilon}} |u_\varepsilon^s|^2\ dx\\
+&\int_{\partial B_R}\frac{\partial u_\varepsilon^s}{\partial n}\cdot \overline{u_\varepsilon^s}\ ds_x-\int_{\partial D_\varepsilon}\frac{\partial u_\varepsilon^s}{\partial n}\cdot\overline{u_\varepsilon^s}\ ds_x=0,
\end{split}
\end{equation}
where and in the sequel, $n=(n_i)_{i=1}^N$ denotes the exterior unit normal vector of the boundary of the concerned domain. Similarly, by multiplying both sides of the second and third equations in (\ref{eq:s2}) by $\overline{u_\varepsilon^-}$ and by using integration by parts over $D_\varepsilon$, we have
\begin{equation}\label{eq:d2}
\begin{split}
&-\int_{D_\varepsilon\backslash\overline{D}_{\varepsilon/2}} (\widetilde{\sigma}_l\nabla u_\varepsilon^-)\cdot(\nabla\overline{u_\varepsilon^-})\ dx+\omega^2\int_{D_\varepsilon\backslash\overline{D}_{\varepsilon/2}}\widetilde{q}_l |u_\varepsilon^-|^2\ dx\\
-&\int_{D_{\varepsilon/2}}(\widetilde{\sigma}_a\nabla u_\varepsilon^-)\cdot(\nabla \overline{u_\varepsilon^-})\ dx+\omega^2\int_{D_{\varepsilon/2}}\widetilde{q}_a|u_\varepsilon^-|^2\ dx\\
+& \int_{\partial D_\varepsilon}\left( \sum_{i,j=1}^N n_i\widetilde{\sigma}_l^{ij}\partial_j u_\varepsilon^- \right)\cdot \overline{u_\varepsilon^-}\ ds_x=0 .
\end{split}
\end{equation}
By adding (\ref{eq:d1}) and (\ref{eq:d2}) and then taking the imaginary parts of both sides of the resultant equation, we further have
\begin{equation}\label{eq:d3}
\begin{split}
& \omega^2\int_{D_\varepsilon\backslash D_{\varepsilon/2}} \Im \widetilde{q}_l |u_\varepsilon^-|^2\ dx+\omega^2\int_{D_{\varepsilon/2}} \Im \widetilde{q}_a |u_\varepsilon^-|^2\ dx\\
=&\Im\bigg\{  -\int_{\partial B_R}\frac{\partial u_\varepsilon^s}{\partial n}\cdot\overline{u_\varepsilon^s}\ ds_x  +\int_{\partial D_\varepsilon} \frac{\partial u_\varepsilon^s}{\partial n}\cdot\overline{u_\varepsilon^s}\ ds_x\\
&-\int_{\partial D_\varepsilon}\bigg(  \sum_{i,j=1}^N n_i\widetilde{\sigma}_l^{ij}\partial_j u_\varepsilon^- \bigg)\cdot\overline{u_\varepsilon^-} ds_x    \bigg\}
\end{split}
\end{equation}
Next, by using the transmission boundary condition on $\partial D_\varepsilon$ (namely, the fourth and fifth equations in (\ref{eq:Helm2})), we have
\begin{equation}\label{eq:d4}
\begin{split}
&\int_{\partial D_\varepsilon}\frac{\partial u_\varepsilon^s}{\partial n}\cdot \overline{u_\varepsilon^s}\ ds_x-\int_{\partial D_\varepsilon} \bigg(  \sum_{i,j=1}^N n_i\widetilde{\sigma}_l^{ij}\partial_j u_\varepsilon^- \bigg)\cdot\overline{u_\varepsilon^-} ds_x  \\
=& -\int_{\partial D_\varepsilon} \left(\sum_{i,j=1}^N n_i\widetilde{\sigma}_l^{ij}\partial_j u_\varepsilon^-\right)\cdot\overline{u^i}\ ds_x-\int_{\partial D_\varepsilon} \frac{\partial u^i}{\partial n}\cdot \overline{u_\varepsilon^-}\ ds_x\\
& +\int_{\partial D_\varepsilon}\frac{\partial u^i}{\partial n}\cdot \overline{u^i}\ ds_x.
\end{split}
\end{equation}
Furthermore, by using the transmission conditions on $\partial D_\varepsilon$ again and the fact that both $u_\varepsilon^s$ and $u^i$ are solutions to $(\Delta+\omega^2)u=0$, we have
\begin{equation}\label{eq:d5}
\begin{split}
& \int_{\partial D_\varepsilon} \frac{\partial u^i}{\partial n}\cdot \overline{u_\varepsilon^-}\ ds_x
-\int_{\partial D_\varepsilon}\frac{\partial u^i}{\partial n}\cdot \overline{u^i}\ ds_x
= \int_{D_\varepsilon}\frac{\partial u^i}{\partial n}\cdot \overline{u_\varepsilon^s}\ ds_x\\
=&\int_{\partial B_R}\frac{\partial u^i}{\partial n}\cdot \overline{u_\varepsilon^s}\ ds_x
-\int_{\partial B_R} u^i\overline{\frac{\partial u_\varepsilon^s}{\partial n}}\ ds_x+\int_{\partial D_\varepsilon} u^i\cdot \overline{\frac{\partial u_\varepsilon^s}{\partial n}}\ ds_x\\
=& \int_{\partial B_R}\frac{\partial u^i}{\partial n}\cdot \overline{u_\varepsilon^s}\ ds_x-\int_{\partial B_R} u^i\cdot \overline{\frac{\partial u_\varepsilon^s}{\partial n}}\ ds_x\\
+& \int_{\partial D_\varepsilon} u^i\cdot \overline{\left(\sum_{i,j=1}^Nn_i\widetilde{\sigma}_l^{ij}\partial_j u_\varepsilon^-\right)}\ ds_x-\int_{\partial D_\varepsilon} u^i\cdot \overline{\frac{\partial u^i}{\partial n}}\ ds_x .
\end{split}
\end{equation}
By combining (\ref{eq:d3})--(\ref{eq:d5}), we have
\begin{equation}\label{eq:d6}
\begin{split}
& \omega^2\int_{D_\varepsilon\backslash D_{\varepsilon/2}} \Im \widetilde{q}_l |u_\varepsilon^-|^s\ dx+\omega^2\int_{D_{\varepsilon/2}} \Im \widetilde{q}_a |u_\varepsilon^-|^2\ dx\\
=& \Im\bigg\{ -\int_{\partial B_R}\frac{\partial u_\varepsilon^s}{\partial n}\cdot \overline{u_\varepsilon^s}\ ds_x-\int_{\partial B_R} \frac{\partial u^i}{\partial n}\cdot \overline{u_\varepsilon^s}\ ds_x\\
&+\int_{\partial B_R} u^i\cdot \overline{\frac{\partial u_\varepsilon^s}{\partial n}}\ ds_x-\int_{\partial D_\varepsilon} u^i\cdot \overline{\frac{\partial u^i}{\partial n}}\ ds_x \bigg\},
\end{split}
\end{equation}
which together with the fact that
\[
\left| \int_{\partial D_\varepsilon} u^i\cdot\overline{\frac{\partial u^i}{\partial n}}\ ds_x \right|=\left| \int_{D_\varepsilon} \Delta\overline{u^i}\cdot u^i\ dx+\int_{D_\varepsilon}|\nabla u^i|^2\ dx \right|\leq C\varepsilon^N\|u^i\|_{H^1(B_R)}^2,
\]
 readily implies (\ref{eq:l2 estimate}).

The proof is complete.
\end{proof}

\begin{lem}\label{lem:main13}
Suppose that 
\[
\alpha(x)\leq \overline{\alpha}_0\quad \mbox{and}\quad \underline{\beta}_0\leq \beta(x)\leq \overline{\beta}_0,\quad x\in D\backslash D_{1/2},
\]
where $\overline{\alpha}_0$, $\underline{\beta}_0$ and $\overline{\beta}_0$ are positive constants.
Let $u_\varepsilon\in H_{loc}^1(\mathbb{R}^N)$ be the solution to (\ref{eq:Helm2}) with $f_\varepsilon=0$. Then we have
\begin{equation}\label{eq:estimate13}
\begin{split}
&\left\|\frac{\partial u_\varepsilon^+}{\partial n}(\varepsilon\ \cdot)\right\|_{H^{-3/2}(\partial D)}^2\\
 & \leq\  C_2 \varepsilon^{2r-2-N}\frac{\left(C_1+\sqrt{\overline{\alpha}_0^2+\overline{\beta}_0^2}\omega^2\varepsilon^{2-r}\right)^2}{\omega^2\underline{\beta}_0}\times\\
 &\bigg( \left\|\Lambda\left(\frac{\partial u_\varepsilon^s}{\partial n}\right)\right\|_{H^{1/2}(\partial B_R)} \left\| \frac{\partial u^i}{\partial n} \right\|_{H^{-1/2}(\partial B_R)} \\
&+ \left\|\frac{\partial u_\varepsilon^s}{\partial n}\right\|_{H^{-1/2}(\partial B_R)}\left\| u^i \right\|_{H^{1/2}(\partial B_R)}+\varepsilon^N\|u^i\|^2_{H^1(B_R)}\\
& +\left\|\frac{\partial u_\varepsilon^s}{\partial n}\right\|_{H^{-1/2}(\partial B_R)} \left\|\Lambda\left(\frac{\partial u_\varepsilon^s}{\partial n}\right)\right\|_{H^{1/2}(\partial B_R)}\bigg),
\end{split}
\end{equation}
where $C_1$ and $C_2$ are the positive constants respectively, from Lemmas~\ref{lem:main11} and \ref{lem:main12}.
\end{lem}

\begin{proof}
Using the fact that
\[
\frac{\partial u_\varepsilon^+}{\partial n}=\sum_{i,j=1}^N n_i\widetilde{\sigma}_l^{ij}\partial_j u_\varepsilon^-\qquad \mbox{on\ \ $\partial D_\varepsilon$},
\]
the lemma follows by combining Lemmas~\ref{lem:main11} and \ref{lem:main12}.
\end{proof}

The next lemma concerns the scattering estimates due to small sound-hard like inclusions.

\begin{lem}\label{lem:sh estimate}
Let $v_\tau\in H_{loc}^1(\mathbb{R}^N\backslash\overline{D}_\tau)$ be the solution to 
\begin{equation}\label{eq:aux1}
\begin{cases}
& \Delta v_\tau+\omega^2 v_\tau=0\hspace*{2.1cm} \mbox{in\ \ $\mathbb{R}^N\backslash\overline{D}_\tau$},\\
& \displaystyle{\frac{\partial v_\tau}{\partial n}=\psi\in H^{-1/2}(\partial D_\tau)\qquad \mbox{on\ \ $\partial D_\tau$}},\\
& \displaystyle{\lim_{|x|\rightarrow\infty}  |x|^{(N-1)/2}\left\{ \frac{\partial v_\tau}{\partial n}-i\omega v_\tau  \right\}=0 }.
\end{cases}
\end{equation}
Then there exists $\tau_0<R$ such that when $\tau<\tau_0$
\begin{equation}\label{eq:case1}
\|v_\tau\|_{H^{1/2}(\partial B_R)}\leq C\tau^{N-1}\|\psi(\tau\ \cdot)\|_{H^{-3/2}(\partial D)},
\end{equation}
and for the particular case if $\psi(x)=\partial u^i/\partial n$ on $\partial D_\tau$,
\begin{equation}\label{eq:case2}
\|v_\tau\|_{H^{1/2}(\partial B_R)}\leq C\tau^{N}\|\psi(\tau\ \cdot)\|_{H^{-1/2}(\partial D)},
\end{equation}
where $C$ is a constant depending only on $\tau_0$, $\omega$, $R$ and $D$ .
\end{lem}

\begin{proof}
The proof of (\ref{eq:case1}) can be modified directly from the proof of Lemma~4.2 in \cite{LiuSun}, and the proof of (\ref{eq:case2}) can be modified directly from that of Lemma~4.1 in \cite{LiuSun}, both based on layer potential techniques. We also refer to the excellent monograph \cite{Ammari4} for related results. 
\end{proof}

We are in a position to present the proof of Theorem~\ref{thm:main1}.

\begin{proof}[Proof of Theorem~\ref{thm:main1}]
Let $v_1\in H_{loc}^{1}(\mathbb{R}^N\backslash \overline{D}_\varepsilon)$ be the scattering solution to 
\begin{equation}\label{eq:aux2}
\begin{cases}
& \Delta v_1+\omega^2 v_1=0\hspace*{2.65cm} \mbox{in\ \ $\mathbb{R}^N\backslash\overline{D}_\varepsilon$},\\
& \displaystyle{\frac{\partial v_1}{\partial n}=\frac{\partial u_\varepsilon^+}{\partial n}\in H^{-1/2}(\partial D_\varepsilon)\qquad \mbox{on\ \ $\partial D_\varepsilon$}},\\
& \displaystyle{\lim_{|x|\rightarrow\infty}  |x|^{(N-1)/2}\left\{ \frac{\partial v_1}{\partial n}-i\omega v_1  \right\}=0 },
\end{cases}
\end{equation}
and let $v_2\in H_{loc}^{1}(\mathbb{R}^N\backslash \overline{D}_\varepsilon)$ be the scattering solution to 
\begin{equation}\label{eq:aux3}
\begin{cases}
& \Delta v_2+\omega^2 v_2=0\hspace*{3.05cm} \mbox{in\ \ $\mathbb{R}^N\backslash\overline{D}_\varepsilon$},\\
& \displaystyle{\frac{\partial v_2}{\partial n}=\frac{\partial u^i}{\partial n}\in H^{-1/2}(\partial D_\varepsilon)\qquad\quad\ \mbox{on\ \ $\partial D_\varepsilon$}},\\
& \displaystyle{\lim_{|x|\rightarrow\infty}  |x|^{(N-1)/2}\left\{ \frac{\partial v_2}{\partial n}-i\omega v_2  \right\}=0 }.
\end{cases}
\end{equation}
Clearly, we have
\[
u_\varepsilon^s=v_1-v_2\qquad \mbox{in\ \ $\mathbb{R}^N\backslash\overline{D}_\varepsilon$}.
\]
By taking $\tau=\varepsilon$ in Lemma~\ref{lem:sh estimate}, we have
\begin{equation}\label{eq:dd1}
\begin{split}
& \|u_\varepsilon^s\|_{H^{1/2}(\partial B_R)} \\
\leq & C_1\bigg( \varepsilon^{N-1}\left\|\left(\frac{\partial u_\varepsilon^+}{\partial n}\right)(\varepsilon\ \cdot)\right \|_{H^{-3/2}(\partial D)}
+\varepsilon^N\left\|\left(\frac{\partial u^i}{\partial n} \right)(\varepsilon\ \cdot)\right\|_{H^{-1/2}(\partial D)}   \bigg )\\
\leq & C_2 \varepsilon^{N-1}\left\|\left(\frac{\partial u_\varepsilon^+}{\partial n}\right)(\varepsilon\ \cdot)\right \|_{H^{-3/2}(\partial D)}+C_2\varepsilon^N\| u^i \|_{H^1(B_R)}.
\end{split}
\end{equation}

Next, we first consider the case with $r\geq 2$. By (\ref{eq:estimate13}), we know that there exists $C_3$ such that
\begin{equation}\label{eq:dd2}
\begin{split}
&\left\| \frac{\partial u_\varepsilon^+}{\partial n}(\varepsilon\ \cdot) \right\|_{H^{-3/2}(\partial D)}\\
\leq & C_3\varepsilon^{1-N/2}\bigg(  \left\|\Lambda\left(\frac{\partial u_\varepsilon^s}{\partial n}\right)\right\|^{1/2}_{H^{1/2}(\partial B_R)}\left\| \frac{\partial u^i}{\partial n} \right\|_{H^{-1/2}(\partial B_R)}^{1/2}   \\ 
+ &\left\|\frac{\partial u_\varepsilon^s}{\partial n}\right\|^{1/2}_{H^{-1/2}(\partial B_R)}\left\| u^i \right\|_{H^{1/2}(\partial B_R)}^{1/2}+\varepsilon^{N/2}\|u^i\|_{H^1(B_R)}\\
+ & \left\|\frac{\partial u_\varepsilon^s}{\partial n}\right\|^{1/2}_{H^{-1/2}(\partial B_R)} \left\|\Lambda\left(\frac{\partial u_\varepsilon^s}{\partial n}\right)\right\|_{H^{1/2}(\partial B_R)}^{1/2}\bigg).
\end{split}
\end{equation}
By (\ref{eq:dd1}) and (\ref{eq:dd2}), we have 
\begin{equation}\label{eq:dd3}
\begin{split}
&\left\|\Lambda\left(\frac{\partial u_\varepsilon^s}{\partial n}\right)\right\|_{H^{1/2}(\partial B_R)}\\
\leq & 4C_2^2C_3^2\varepsilon^N\left\| \frac{\partial u^i}{\partial n} \right\|_{H^{-1/2}(\partial B_R)}+\frac 1 4 \left\|\Lambda\left(\frac{\partial u_\varepsilon^s}{\partial n}\right)\right\|_{H^{1/2}(\partial B_R)}\\
+& C_2C_3\varepsilon^{N/2}\left\| \frac{\partial u_\varepsilon^s}{\partial n} \right\|_{H^{-1/2}(\partial B_R)}^{1/2}\|u^i\|_{H^{1/2}(\partial B_R)}^{1/2}\\
+&4C_2^2C_3^2\varepsilon^N\left\| \frac{\partial u_\varepsilon^s}{\partial n} \right\|_{H^{-1/2}(\partial B_R)}
+\frac 1 4 \left\|\Lambda\left(\frac{u_\varepsilon^s}{\partial n}\right)\right\|_{H^{1/2}(\partial B_R)}\\
+& (C_2C_3+C_2)\varepsilon^N\|u^i\|_{H^1(B_R)},
\end{split}
\end{equation}
and from which we further have that there exists $C_4$ such that 
\begin{equation}\label{eq:dd4}
\begin{split}
& \left\|\Lambda\left(\frac{\partial u_\varepsilon^s}{\partial n}\right)\right\|_{H^{1/2}(\partial B_R)}\\
\leq & C_4\varepsilon^{N/2}\left\|\frac{\partial u_\varepsilon^s}{\partial n}\right\|_{H^{-1/2}(\partial B_R)}^{1/2}\|u^i\|_{H^{1/2}(\partial B_R)}^{1/2}+C_4\varepsilon^N\left\| \frac{\partial u_\varepsilon^s}{\partial n} \right\|_{H^{-1/2}(\partial B_R)}\\
&+C_4\varepsilon^N\| u^i \|_{H^1(B_R)}\\
\leq & \frac{C_4^2}{\epsilon}\varepsilon^N\|u^i\|_{H^{1/2}(\partial B_R)}+\epsilon\left\|\frac{\partial u_\varepsilon^s}{\partial n}\right\|_{H^{-1/2}(\partial B_R)}\\
&+C_4\varepsilon^N\left\| \frac{\partial u_\varepsilon^s}{\partial n} \right\|_{H^{-1/2}(\partial B_R)}+C_4\varepsilon^N\| u^i \|_{H^1(B_R)},
\end{split}
\end{equation}
where $\epsilon\in\mathbb{R}_+$.
By choosing $\varepsilon_0$ sufficiently small such that when $\varepsilon<\varepsilon_0$,
\[
C_4\varepsilon^N\leq \frac 1 4 \|\Lambda^{-1}\|^{-1}_{\mathcal{L}(H^{1/2}(\partial B_R), H^{-1/2}(\partial B_R))}
\]
and also by choosing $\epsilon$ in (\ref{eq:dd4}) such that
\[
\epsilon\leq \frac 1 4 \|\Lambda^{-1}\|^{-1}_{\mathcal{L}(H^{1/2}(\partial B_R), H^{-1/2}(\partial B_R))},
\]
one can show from (\ref{eq:dd4}) by straightforward calculations that there exists $C$ such that
\begin{equation}\label{eq:dd5}
\left\|\frac{\partial u_\varepsilon^s}{\partial n}\right\|_{H^{-1/2}(\partial B_R)}\leq C\varepsilon^N\|u^i\|_{H^1(B_R)}.
\end{equation}
Since $A_\varepsilon(\hat{x})$ could be read from the large $|x|$ asymptotics of $u_\varepsilon^s$, by the well-posedness of the forward scattering problem for $u_\varepsilon^s$ in the exterior of $B_R$, one readily has
\[
|A_\varepsilon(\hat{x})|\leq C\varepsilon^N\|u^i\|_{H^1(B_R)},\qquad \forall \hat{x}\in\mathbb{S}^{N-1}.
\]

Next, we consider the case with $2-N/2<r<2$. In this case, by (\ref{eq:estimate13}), one has that there exists $C_5$ such that
\begin{equation}
\begin{split}
&\left\| \frac{\partial u_\varepsilon^+}{\partial n}(\varepsilon\ \cdot) \right\|_{H^{-3/2}(\partial D)}\\
\leq& C_5\varepsilon^{r-1-N/2}\bigg(  \left\|\Lambda\left(\frac{\partial u_\varepsilon^s}{\partial n}\right)\right\|^{1/2}_{H^{1/2}(\partial B_R)}\left\| \frac{\partial u^i}{\partial n} \right\|_{H^{-1/2}(\partial B_R)}^{1/2}   \\ 
&+\left\|\frac{\partial u_\varepsilon^s}{\partial n}\right\|^{1/2}_{H^{-1/2}(\partial B_R)}\left\| u^i \right\|_{H^{1/2}(\partial B_R)}^{1/2}+\varepsilon^{N/2}\|u^i\|_{H^1(B_R)}\\
&+ \left\|\frac{\partial u_\varepsilon^s}{\partial n}\right\|^{1/2}_{H^{-1/2}(\partial B_R)} \left\|\Lambda\left(\frac{\partial u_\varepsilon^s}{\partial n}\right)\right\|_{H^{1/2}(\partial B_R)}^{1/2}\bigg).
\end{split}
\end{equation}
which in combination with (\ref{eq:dd1}) yields
\begin{equation}\label{eq:ddd1}
\begin{split}
&\left\| \Lambda\left(\frac{\partial u_\varepsilon^s}{\partial n}\right) \right\|_{H^{1/2}(\partial B_R)}\\
\leq & {C}_6\varepsilon^{r-2+N/2}\bigg(  \left\|\Lambda\left(\frac{\partial u_\varepsilon^s}{\partial n}\right)\right\|^{1/2}_{H^{1/2}(\partial B_R)}\left\| \frac{\partial u^i}{\partial n} \right\|_{H^{-1/2}(\partial B_R)}^{1/2}   \\ 
+ &\left\|\frac{\partial u_\varepsilon^s}{\partial n}\right\|^{1/2}_{H^{-1/2}(\partial B_R)}\left\| u^i \right\|_{H^{1/2}(\partial B_R)}^{1/2}+\varepsilon^{N/2}\|u^i\|_{H^1(B_R)}\\
+ & \left\|\frac{\partial u_\varepsilon^s}{\partial n}\right\|^{1/2}_{H^{-1/2}(\partial B_R)} \left\|\Lambda\left(\frac{\partial u_\varepsilon^s}{\partial n}\right)\right\|_{H^{1/2}(\partial B_R)}^{1/2}\bigg)\\
&+C_6\varepsilon^N\|u^i\|_{H^1(B_R)}.
\end{split}
\end{equation}
By a similar algebraic argument as earlier, one can show from (\ref{eq:ddd1}) that
\[
\left\|\frac{\partial u_\varepsilon^s}{\partial n}\right\|_{H^{-1/2}(\partial B_R)}\leq C\varepsilon^{N+2r-4}\|u^i\|_{H^1(B_R)},
\]
which in turn implies
\[
|A_\varepsilon(\hat{x})|\leq C\varepsilon^{N+2r-4}\|u^i\|_{H^1(B_R)}.
\]

The proof is completed.
\end{proof}

Next, we consider the cloaking of an active/radiating source term by assuming that $f\in L^2(D_{1/2})$ in (\ref{eq:Helm1}). We shall show that

\begin{thm}\label{thm:main2}
Suppose $f\in L^2(D_{1/2})$ and $\sigma_a, \Re q_a$ are arbitrary but regular, and
\begin{equation}\label{eq:lossy cloaked}
\Im q_a\geq q_0>0\quad \mbox{on\ \ $supp(f)\subset D_{1/2}$}.
\end{equation}
Let $A_\varepsilon(\hat{x})$ be the scattering amplitude to (\ref{eq:Helm2}). Let $B_R$, $R\in\mathbb{R}_+$, be a central ball of radius $R$ such that $\Omega\subset B_R$. Then there exists $\varepsilon_0\in\mathbb{R}_+$ such that when $\varepsilon<\varepsilon_0$
\begin{equation}\label{eq:estimate main 2}
\begin{split}
|A_\varepsilon(\hat{x})|\leq & C\bigg( \varepsilon^{\min\{N+2r-4, N\}} \|u^i\|_{H^1(B_R)} \\
 & + \varepsilon^{\min\{N/2, N/2+r-2 \}}\|f\|_{L^2(D_{1/2})}     \bigg ),\quad \forall\hat{x}\in\mathbb{S}^{N-1}
\end{split}
\end{equation}
where $C$ is positive constant independent of $\varepsilon$, $\sigma_a$, $\Re q_a$, $r$ and $f$.
\end{thm}

\begin{rem}\label{rem:active}
We first note that if one takes $f=0$ in Theorem~\ref{thm:main2}, then Theorem~\ref{thm:main1} is recovered. Next, we shall emphasize the critical role of the lossy layer. We only consider the two-dimensional case as an example. If one excludes the lossy layer by taking $\widetilde\sigma_l=\widetilde\sigma_a$ and $\widetilde{q}_l=\widetilde{q}_a$ to be arbitrary (but regular) as part of the passive content being cloaked. Let $\widetilde\sigma_a=I$, $\widetilde q_a=1$ in $D_{\varepsilon}$, or equivalently, $\sigma_a=I$, $q_a=\varepsilon^2$ in $D$. Moreover, we let $f=c_0\chi_{D_{1/2}}$ be the source term that one intends to cloak, where $c_0$ is a generic positive constant. It is readily seen that $(\sigma_a, q_a)$ is not the `resonant' inclusion discussed for passive cloaking in Remark~\ref{rem:passive}. However, it can be straightforwardly shown that generically one would have a significant $A_\varepsilon(x)$ for the virtual scattering problem \eqref{eq:Helm2} with the parameters specified in the above, i.e., one cannot cloak the active source term $f$. For our near-cloaking construction, it can be seen from Theorem~\ref{thm:main2} that if one maintains the place where the source term is located to be absorbing, namely condition \eqref{eq:lossy cloaked} is satisfied, then a much practical and favorable near-cloak could be achieved. It is emphasized that in Theorem~\ref{thm:main2} we only proposed one possible way of effectively cloaking a region with active/radiating contents, but we do not claim that it is the most efficient way. We would also like to note that in \cite{LasZho}, the cloaking of a source term is also considered, but the study there is of different interests. 
\end{rem}

In order to prove Theorem~\ref{thm:main2}, we shall first derive a lemma similar to Lemma~\ref{lem:main12}.

\begin{lem}\label{lem:main22}
Suppose that 
\[
\alpha(x)\leq \overline{\alpha}_0\quad \mbox{and}\quad \underline{\beta}_0\leq \beta(x)\leq \overline{\beta}_0,\quad x\in D\backslash D_{1/2},
\]
where $\overline{\alpha}_0$, $\underline{\beta}_0$ and $\overline{\beta}_0$ are positive constants, and
\begin{equation}\label{eq:ll}
\Im q_a\geq q_0>0\quad \mbox{on\ \ $supp(f)\subset D_{1/2}$}.
\end{equation}
Let $u_\varepsilon\in H_{loc}^1(\mathbb{R}^N)$ be the solution to (\ref{eq:Helm2}). Then we have
\begin{equation}\label{eq:2l2 estimate}
\begin{split}
&\omega^2\underline{\beta}_0\int_{D_\varepsilon\backslash D_{\varepsilon/2}} |u_\varepsilon^-|^2\ d x\\
 \leq &\ C\bigg( \left\|\Lambda\left(\frac{\partial u_\varepsilon^s}{\partial n}\right)\right\|_{H^{1/2}(\partial B_R)} \left\| \frac{\partial u^i}{\partial n} \right\|_{H^{-1/2}(\partial B_R)} \\
&+ \left\|\frac{\partial u_\varepsilon^s}{\partial n}\right\|_{H^{-1/2}(\partial B_R)}\left\| u^i \right\|_{H^{1/2}(\partial B_R)}+\varepsilon^N\|u^i\|_{H^1(B_R)}^2\\
& +\left\|\frac{\partial u_\varepsilon^s}{\partial n}\right\|_{H^{-1/2}(\partial B_R)} \left\|\Lambda\left(\frac{\partial u_\varepsilon^s}{\partial n}\right)\right\|_{H^{1/2}(\partial B_R)}+\|f\|_{L^2(D_{1/2})}^2\bigg),
\end{split}
\end{equation}
where $C$ is a positive constant depending only on $R$, $\omega$ and $q_0$. 
\end{lem}

\begin{proof}
W.L.O.G, we assume that 
\[
supp(f)=D_{1/2}.
\]
By a similar argument to that for the proof of Lemma~\ref{lem:main12}, one can show by straightforward calculations that 
\begin{equation}\label{eq:2d6}
\begin{split}
& \omega^2\int_{D_\varepsilon\backslash D_{\varepsilon/2}} \Im \widetilde{q}_l |u_\varepsilon^-|^2\ dx+\omega^2\int_{D_{\varepsilon/2}} \Im \widetilde{q}_a |u_\varepsilon^-|^2\ dx\\
=& \Im\bigg\{ -\int_{\partial B_R}\frac{\partial u_\varepsilon^s}{\partial n}\cdot \overline{u_\varepsilon^s}\ ds_x-\int_{\partial B_R} \frac{\partial u^i}{\partial n}\cdot \overline{u_\varepsilon^s}\ ds_x\\
&+\int_{\partial B_R} u^i\cdot \overline{\frac{\partial u_\varepsilon^s}{\partial n}}\ ds_x-\int_{\partial D_\varepsilon} u^i\cdot \overline{\frac{\partial u^i}{\partial n}}\ ds_x +\int_{D_{\varepsilon/2}}f_\varepsilon u_\varepsilon^-\ dx \bigg\}.
\end{split}
\end{equation}
Using \eqref{eq:ll} in \eqref{eq:2d6}, one further has by direct verifications that 
\begin{equation}\label{eq:2d7}
\begin{split}
&\omega^2\underline{\beta}_0\int_{D_\varepsilon\backslash D_{\varepsilon/2}} |u_\varepsilon^-|^2\ d x+\omega^2 q_0\varepsilon^{-N}\int_{D_{\varepsilon/2}} |u_\varepsilon^-|^2\ dx \\
\leq & \ C\bigg(  \left\|\Lambda\left(\frac{\partial u_\varepsilon^s}{\partial n}\right)\right\|_{H^{1/2}(\partial B_R)}\left\| \frac{\partial u^i}{\partial n} \right\|_{H^{-1/2}(\partial B_R)}   \\ 
+ &\left\|\frac{\partial u_\varepsilon^s}{\partial n}\right\|_{H^{-1/2}(\partial B_R)}\left\| u^i \right\|_{H^{1/2}(\partial B_R)}+\varepsilon^{N}\|u^i\|_{H^1(B_R)}^2\\
+ & \left\|\frac{\partial u_\varepsilon^s}{\partial n}\right\|_{H^{-1/2}(\partial B_R)} \left\|\Lambda\left(\frac{\partial u_\varepsilon^s}{\partial n}\right)\right\|_{H^{1/2}(\partial B_R)}\\
+& \frac{\varepsilon^{-N}}{\epsilon}\|f\left(\frac 1\varepsilon\ \cdot \right)\|_{L^2(D_{\varepsilon/2})}^2+\epsilon\varepsilon^{-N}\|u_\varepsilon^-\|_{L^2(D_{\varepsilon/2})}^2\bigg)
\end{split}
\end{equation}
By choosing $\epsilon=\omega^2q_0/2$, together with the use of the fact that
\[
\|f(\frac{1}{\varepsilon}\ \cdot)\|_{L^2(D_{\varepsilon/2})}=\varepsilon^{N/2}\|f\|_{L^2(D_{1/2})},
\]
one has (\ref{eq:2l2 estimate}) by straightforward verifications.

The proof is completed.
\end{proof}

\begin{proof}[Proof of Theorem~\ref{thm:main2}]
By Lemmas~\ref{lem:main11}, \ref{lem:sh estimate} and \ref{lem:main22} and a similar argument to that for the proof of Theorem~\ref{thm:main1}, one can show that
\begin{equation}\label{eq:2ddd1}
\begin{split}
&\left\| \Lambda\left(\frac{\partial u_\varepsilon^s}{\partial n}\right) \right\|_{H^{1/2}(\partial B_R)}\\
\leq\ & {C} \varepsilon^{\min\{r-2+N/2, N/2\}}
\bigg(  \left\|\Lambda\left(\frac{\partial u_\varepsilon^s}{\partial n}\right)\right\|^{1/2}_{H^{1/2}(\partial B_R)}\left\| \frac{\partial u^i}{\partial n} \right\|_{H^{-1/2}(\partial B_R)}^{1/2}   \\ 
&+\left\|\frac{\partial u_\varepsilon^s}{\partial n}\right\|^{1/2}_{H^{-1/2}(\partial B_R)}\left\| u^i \right\|_{H^{1/2}(\partial B_R)}^{1/2}+\varepsilon^{N/2}\|u^i\|_{H^1(B_R)}\\
&+ \left\|\frac{\partial u_\varepsilon^s}{\partial n}\right\|^{1/2}_{H^{-1/2}(\partial B_R)} \left\|\Lambda\left(\frac{\partial u_\varepsilon^s}{\partial n}\right)\right\|_{H^{1/2}(\partial B_R)}^{1/2}+\|f\|_{L^2(D_{1/2})}\bigg).
\end{split}
\end{equation}
Using a similar algebraic argument to that for the proof of Theorem~\ref{thm:main1}, one can show by direct calculations that
\[
\begin{split}
&\left\|\frac{\partial u_\varepsilon^s}{\partial n}\right\|_{H^{-1/2}(\partial B_R)}\\
\leq& \  C\bigg( \varepsilon^{\min\{N+2r-4, N\}} \|u^i\|_{H^1(B_R)} 
 + \varepsilon^{\min\{N/2, N/2+r-2 \}}\|f\|_{L^2(D_{1/2})}     \bigg )
 \end{split}
\]
which immediately implies (\ref{eq:estimate main 2}) and completes the proof.
\end{proof}

\section{Conclusion and discussion}

In this work, we have been mainly concerned with the near-invisibility cloaking by the `blow-up-a-small-region' construction through the transformation optics. From a practical viewpoint, we mainly considered the case that the cloaked content is arbitrary, and moreover it could be both a passive medium or an active/radiating source. However, there are cloaking-busting inclusions which defy the attempt to achieve near-cloaks for both passive cloaking and active cloaking. In order to defeat the cloak-busts, a lossy layer is incorporated into our construction. Such a damping mechanism was originated in \cite{KOVW} by using a special layer with a high loss parameter, and was later adopted in \cite{LiuSun} by using a layer with a high density parameter. In the present paper, the lossy layer in our scheme is very general which could be variable, and even anisotropic, and the density parameter ranges from very high to reasonably low. This provides more flexibility in the construction of practical cloaking devices. We assessed the cloaking performance for our construction in terms of the scattering amplitude due to a time-harmonic wave.  We derive very accurate estimates of the scattering amplitude in terms of the regularization parameter and the material parameters of the lossy layer disregarding the cloaked contents. It is worth noting that in Theorem~\ref{thm:main2}, we provide an effective way in cloaking active contents by maintaining the place where the active object is located to be absorbing. 

\section*{Acknowledgement}

The work of Hongyu Liu is supported by NSF grant, DMS 1207784. The helpful comments from the anonymous referee are gratefully acknowledged.

\end{document}